\documentclass[reqno,12pt]{amsart}

\usepackage{amsfonts, amssymb, verbatim}
\usepackage{amssymb,epsfig,amscd, verbatim}
\usepackage{amsthm}
\usepackage{amsmath}
\usepackage[T1]{fontenc}

\newtheorem{theorem}{Theorem}
\newtheorem{lemma}[theorem]{Lemma}
\newtheorem{proposition}[theorem]{Proposition}

\theoremstyle{definition}
\newtheorem{definition}[theorem]{Definition}

\numberwithin{equation}{section} \numberwithin{theorem}{section}

\theoremstyle{remark}
\newtheorem{remark}[theorem]{Remark}

\def\z{\,_{\dot z}\,}

\def\vac{|0\rangle}                            

\def\l{\lambda}
\def\m{\mu}

\def\tt{\otimes}                               
\def\<{\langle}
\def\>{\rangle}

\def\d{\partial}

\def\dprod{\displaystyle\prod}
\def\dsum{\displaystyle\sum}
\def\vacuum{|0\rangle}
\def\vac{\mathbf{1}}                            

\def\z{\,_{\dot z}\,}

\def\ddx{\frac{d}{dx}\,}

\def\ddy{\frac{d}{dy}\,}
\def\ddz{\frac{d}{dz}\,}

\def\dxyz{z^{-1}\delta\left(\frac{x-y}{z}\right)}
\def\dyxmz{z^{-1}\delta\left(\frac{y-x}{-z}\right)}
\def\dxzy{y^{-1}\delta\left(\frac{x-z}{y}\right)}

\def\va{(V,\z,\vac)}
\def\v2a{(V,\z,\vac,D)}

\def\dxyz{\delta\left(\frac{x-y}{z}\right)}
\def\dmyxz{\delta\left(\frac{-y+x}{z}\right)}
\def\dxzy{\delta\left(\frac{x-z}{y}\right)}

\def\va{(V,\z,\vac)}
\def\va1{(V,Y,\vacuum, T)}
\def\v2a{(V,\z,\vac,D)}

\begin{document}

\title{Fusion rules for the Virasoro algebra of central charge 25.}
\author[Fusion Rules ]{Florencia Orosz Hunziker}
\address{Department of Mathematics\\
Yale University\\
	442 Dunham Lab\\
	10 Hillhouse Ave
	New Haven, CT 06511. Fax 203-432-7316, Phone 203-432-7058}

\email{florencia.orosz@yale.edu}

\maketitle
\begin{abstract}
Let $\mathcal{F}_{25}$ be the family of irreducible lowest weight modules for the Virasoro algebra of central charge $25$  which are not isomorphic to Verma modules. Let $L(25,0)$ be the Virasoro vertex operator algebra of central charge 25. We prove that the fusion rules for the $L(25,0)$-modules in $\mathcal{F}_{25}$  are in correspondence with the tensor rules for the irreducible finite dimensional representations of $sl(2, \mathbb{C})$, extending the known correspondence between modules for the Virasoro algebras of dual central charges 1 and 25. 	
\end{abstract}
\section{introduction}
In 1990, in their important paper \cite{FF}, Feigin and Fuchs described the structure of Verma modules for the Virasoro algebra  and the homomorphisms between them. They stated the projection formulas for singular vectors on the density modules and described the duality between the category of Verma modules with central charge $c$ and the category of Verma modules with central charge $26-c$ for $c\in \mathbb{C}$. In particular,  they established an anti-equivalence of additive categories between the category of Verma modules for $Vir_{c=1}$ and the Verma modules for $Vir_{c=25}$ which assigns the Verma module of central charge $1$ and lowest weight $h$, $M(1, h)$, to the Verma module of central charge $25$ and lowest weight $1-h$, $M(25, 1-h)$ and reverses morphisms. On the other hand, Segal on his 1981 paper \cite{Se} noted a correspondence between the finite dimensional irreducible representations of $sl(2,\mathbb{C})$ and certain representations of the Virasoro algebra of central charge 1 by a dual pair type of argument. 
Later, Frenkel and Zhu, proved in \cite{FZ} that for $c\in \mathbb{C}$, $L(c,0)$, the irreducible quotient of the Verma module $M(c,0)$, has a vertex algebra operator algebra structure and that the irreducible quotient of $M(c,h)$, $L(c, h)$,  is an $L(c,0)$-module for $h\in\mathbb{C}$. The vertex algebra version of the decomposition noted in \cite{Se} was proved by Dong and Griess in \cite{DG}.
In his doctorate thesis \cite{S}, Styrkas used the dual pair decomposition of Segal to state an antiequivalence of tensor categories between the category of irreducible finite dimensional modules for $sl(2, \mathbb{C})$ and the semisimple tensor category generated by the irreducible $L(1,0)$-modules which are not Verma modules, i.e the modules of the from $L\left(1,\frac{n^2}{4}\right)$ for $n\geq 0$.
Milas, independently proved in  \cite{M}, that the fusion rules for the irreducible non verma $L(1,0)$ modules coincide with the tensor rules of the irreducible finite dimensional representations of $sl(2,\mathbb{C})$. 
More recently, McRae proved in \cite{Mc} using the fusion rules from \cite{M} that the semisimple category generated by the $L(1,0)$-modules $L(1,\frac{m^2}{4})$ for $m\geq 0$ is equivalent to the tensor structure of finite dimensional irreducible $sl(2,\mathbb{C})$-modules modified by a 3-cocycle.

In this work,  following the methods in \cite{M} and \cite{FZ2} together with results in \cite{Z}, \cite{L}, \cite{W}  we show that the fusion rules for the non-Verma irreducible modules for the Virasoro algebra $L(25,0)$ are the same as the tensor rules for the finite dimensional irreducible representations of $sl(2, \mathbb{C})$. Our main result is  
\begin{theorem}
	Let $m,n\geq 0$. Then 
	\begin{align*}
	{\rm dim \ } I  \binom{L(25, 1-\frac{(r+2)^2}{4})}{L(25, 1-\frac{(m+2)^2}{4})\ \ L(25,1-\frac{(n+2)^2}{4})}=
	\begin{cases} 
	1 & \textrm{\rm if } r \in \{|m-n|, |m-n|+2, \cdots, m+n\}  \\       
	0 & \textrm{\rm otherwise.}  
	\end{cases}
	\end{align*}	
	
\end{theorem}	
The fact that the fusion rules for this distinguished family of $L(25,0)$-modules coincide with those of $sl(2,\mathbb{C})$ generalizes the existing duality between the representations for $L(1,0)$ and $L(25,0).$ Namely, our result shows that the relation between the Virasoro algebras of central charge 1 and 25 is deeper than simply a duality of additive categories, as illustrated in Figures \ref{fig:minipage1}  and \ref{fig:minipage2} below. 
\begin{figure}[ht]
	\centering
	\begin{minipage}[b]{\linewidth}
		\centering
		\includegraphics[scale=1]{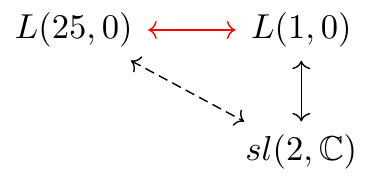}
		\caption{Dualities between $sl(2,\mathbb{C})$ and the Virasoro vertex algebras of central charges 1 and 25.}
		\label{fig:minipage1}
	\end{minipage}

	\quad
	\begin{minipage}[b]{\linewidth}
		\centering
		\includegraphics[scale=1]{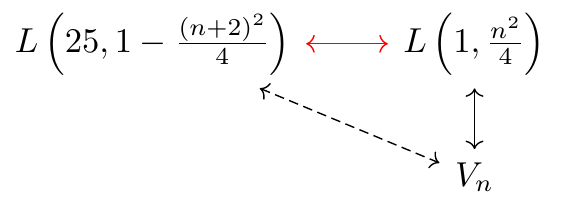}
		\caption{Correspondence of modules in the dualities between representations of $L(25,0)$, $L(1, 0)$ and $sl(2,\mathbb{C})$.}
		\label{fig:minipage2}
	\end{minipage}
\end{figure}

 Our theorem suggests that this equivalence can be extended to an equivalence of semisimple tensor categories.  This tensor category equivalence will be described in a future paper. In \cite{M1} and \cite{MA} logarithmic extensions of intertwining operators between the $L(1,0)$-modules were studied and this technique can be adapted to the case of central charge 25 with the aim of applying the logarithmic tensor product theory \cite{HLZ} to the category $C_{L(25,0)}.$ Moreover, in light of the corresponding Feigin-Fuchs duality \cite{D}, it is an interesting question to explore whether this equivalence of fusion rules and tensor categories holds for general $W$-algebras.

{\bf Aknowledgements} 
The author would like to thank Igor Frenkel for suggesting the study of the case of the Virasoro algebra of central charge 25. The author would also like to thank Igor Frenkel, Jinwei Yang, Antun Milas and Gregg Zuckerman for useful conversations and comments. Finally, the author thanks the reviewer for her/his/their insightful remarks.    
\section{preliminaries}
\subsection{The Virasoro algebra}
The Virasoro Lie algebra, denoted $Vir$ throughout this work, is the complex Lie algebra with generators $\{C, \{L_n\}_{n\in \mathbb{Z}}\}$ and relations
\begin{align*}
[L_n, L_m]&=(n-m)L_{m+n}+\frac{(n^3-n)}{12}\delta_{n,-m}C\\
[L_n,C]&=0.
\end{align*}

\subsubsection{Verma modules}
For complex numbers $c$ and $h$, the Verma module $M(c,h)$ for the Virasoro algebra is defined as 
\begin{align*}
M(c,h):=U(Vir)\tt_{U(Vir^{\geq 0})}\mathbb{C}_c^{h},
\end{align*}
where the $Vir^{\geq 0}:=\bigoplus_{n\geq 0} \mathbb{C}L_n\oplus \mathbb{C}C$, and the $Vir^{\geq0}$-module structure of the one dimensional space $\mathbb{C}_c^{h}:=\mathbb{C}\vac_{c,h}$ is given by

\begin{align*} 
\ &L_n \vac_{c,h}=0 \textrm{ for $n> 0$},\\
\ &L_0 \vac_{c,h}=h\vac_{c,h},\\
\ &C\vac_{c,h}=c \vac_{c,h}.
\end{align*}
We denote by $J(c,h)$ the unique (possibly trivial) maximal submodule of $M(c,h)$  and by
\begin{align*} 
L(c,h):=M(c,h)/J(c,h)
\end{align*}
its irreducible quotient.
On the other hand, the contragredient module $M(c,h)^{'}$ is defined as
\begin{align*}
M(c,h)^{'}:=\bigoplus_{n\in \mathbb{Z}}{\rm Hom}(M(c,h)^{n}, \mathbb{C})
\end{align*}
where $M(c,h)^{n}:=\{ v \in M(c,h) \mid L_0v=nv\}.$

\subsubsection{Density modules} \label{F}
For each $\lambda$ and $\mu$ the density module for the Virasoro algebra $\mathcal{D}_{\l, \m}$ is the module spanned by vectors $\{w_r\}_{r \in \mathbb{Z}}$  with action given by 
\begin{align*}
L_n.w_r=(\mu +r +\l (n+1))w_{r-n}.
\end{align*}

\subsubsection{Singular vectors}\label{svff}
We recall the results on projection formulas for singular vectors on density modules obtained in \cite{FF} following the exposition in \cite{Ke}:

In $M(c,h)$ there is a singular vector of level $N$ if and only if there exists $p,q\in \mathbb{Z}_+$ such that $pq=N$ and $t \in \mathbb{C}$ such that 
\begin{align}
&c=c(t)=13-6t-6t^{-1}\label{ct}\\
&h=h_{p,q}(t)=\frac{p^2-1}{4}t-\frac{pq-1}{2}+\frac{q^2-1}{4}t^{-1}.\nonumber
\end{align}
If $v_{p,q}(t)$ is the singular vector on $M(c,h)$ then $v_{p.q}(t)=O_{p,q}(t)\vac_{c,h}$ with
\begin{align*}
O_{p,q}(t)=\dsum_{|I|=pq}a_{I}^{p,q}(t)L_{-I}
\end{align*}
where $L_{-I}=L_{-i_{1}}\cdots L_{-i_{n}}$ and $|I|=i_1+\cdots i_n$. In this formula the sum is over sequences $I=\{i_1, \cdots i_n\}$ of ordered tuples $i_1\geq \cdots\geq i_n$, each coefficient $a_{I}^{p,q}(t)$ depends polynomially on $t$ and $t^{-1}$ and by convention the coefficient of $(L_{-1})^{pq}$ is 1. Consider $\mathcal{D}_{\l, \m}$ the density module with basis $\{w_r\}_{r \in \mathbb{Z}}$ as in Section \ref{F}. We define the function $f(\lambda, \mu,t)$ by
\begin{align*}
(O_{p,q}(t))w_0=f_{p,q}(\lambda, \mu, t)w_{pq}
\end{align*}
in $\mathcal{D}_{\lambda, \mu}$. Then, we have the following result proved in \cite{FF}:
\begin{proposition}
Let $\theta=\sqrt{-t^{-1}}$ and let
\begin{align*}
A_{p,q}(l,k)=\left(\left(\frac{p-1}{2}+l\right)\theta^{-1}+\left(\frac{q-1}{2}+k\right)\theta \right)\left(\left(\frac{p+1}{2}-l\right)\theta^{-1}+\left(\frac{q+1}{2}-k\right)\theta\right).
\end{align*}
Then, 
\begin{align}
f_{p,q}^2(\lambda, \mu, t)=\prod_{\substack{ -\frac{p-1}{2} \leq l\leq \frac{p-1}{2} \\-\frac{q-1}{2} \leq k \leq \frac{q-1}{2} }} \left(\left(\mu+A_{p,q}(l, k)\right)\left(\mu +A_{p,q}(-l, -k)\right)-4\lambda (l\theta^{-1}+k\theta)^2\right).\label{apq}
\end{align}
\end{proposition}

\subsection{Vertex algebras}

For a $\mathbb{C}$-algebra $R$, we will denote by $R((z))$ the space of Laurent series, namely
\begin{align*}
R((z))=\left\{ \dsum_{l\in \mathbb{Z}} a_l z^{l} \ | \ a_l\neq 0 \textrm{ for finitely many negative $l$}\right\}.
\end{align*}
On the other hand, if $W$ is a $\mathbb{C}$-vector space, we denote by $W\{z\}$ the space of $W$-valued formal series involving the rational powers of $z$. Namely,
\begin{align*}
W\{z\}=\left\{\dsum_{n\in \mathbb{Q}} w_n z^{n}   \ | \ w_n\in W \right\}.
\end{align*} 
\begin{definition}
	Let $V$ be a vector space $\mathbb{C}$ and ${\rm End}V$ the algebra of linear operators on $V.$
	We say that a power series
	\begin{align*}
	a(z)=\dsum_{j\in \mathbb{Z}}a_jz^{-j} \in {\rm End}V[[z^{\pm 1}]]
	\end{align*}
	is a {\it field} if for any $v \in V$ $a_jv=0$ for $j>>0.$ That is, for each $v \in V$,
	\begin{align*}
	a(z)v \in V((z)).
	\end{align*}
\end{definition}
\begin{definition}\cite{FLM}
	A vertex algebra $\va1$ consists of a $\mathbb{C}$-vector space $V$,
a distinguished vector $\vacuum \in V$ called the vacuum vector,
 an operator $T:V\longrightarrow V$  and a linear map \
		\begin{align*}
		Y(\ .\ , z): V \longrightarrow {\rm End\ } V[[z^{\pm 1}]]
		\end{align*}
		taking each $a \in V$ to a field acting on $V,$
		\begin{align*}
		Y(a,z)=\dsum_{n\in \mathbb{Z}}a_nz^{-n-1}.
		\end{align*}

	These data are subject to the following axioms for all $a, b \in V$:\\
	\begin{enumerate}
		\item Vacuum axioms
		\begin{align*}
		&Y(\vacuum, z)=Id_V\\
		&Y(a,z) \vacuum \in V[[z]] \ \ \textrm{ and } \ \ Y(a,z)\vacuum|_{z=0}=a,
		\end{align*}
		\item Translation axioms
		\begin{align*}
		&[T,Y(a,z)]=\ddz Y(a,z)\\
		&T\vacuum=0
		\end{align*}
		\item The Jacobi identity 
		\begin{align}
		z^{-1}\dxyz Y(a,x)&Y(b,y)c-z^{-1} \dmyxz Y(b,y)Y(a,x)c=\nonumber \\
		&y^{-1}\dxzy Y(Y(a,z)b,y)c\label{jacobi}.
		\end{align}
	\end{enumerate}
\end{definition} 	
\begin{definition}
	A vertex operator algebra $V$ is a $\mathbb{Z}$-graded vertex algebra
	\begin{align*}
	V=\coprod_{n \in \mathbb{Z}} V_n \ \ \textrm{ with ${\rm dim}V_n<\infty$, $n={\rm wt} v$ and $V_n=0$ for n sufficiently small,}
	\end{align*}
	together with a homogeneous conformal vector $\omega\in V_2$ satifying the following:\\
	If $L(n):=\omega_{n+1}$ for $n \in \mathbb{Z}$ so that  $Y(\omega,z)=\dsum_{n\in \mathbb{Z}}L(n) z^{-n-2}$ then,
	\begin{align*}
	[L(m),L(n)]=(m-n)L(m+n)+\frac{1}{12}(m^3-m)\delta_{m+n,0}c
	\end{align*}
	for $c \in \mathbb{Q}$ which we call the rank of $V.$ Moreover, the action  of the Virasoro algebra is compatible with the gradation and the action of $T$. Namely, $L(0)v=nv=({\rm wt}v)v$ for $n\in \mathbb{Z}$, $v\in V_n$ and $L(-1)=T.$ The conformal dimension of an operator $Y(a,z)$ is $m$ if $a\in V_m.$ Namely, $a_{n} |_{V_l}\subset V_{l+m-n-1}$
\end{definition}

\begin{definition}
	A module $W$ for a vertex operator algebra $V$ is a $\mathbb{Q}$-graded vector space
	\begin{align}
	W=\coprod_{n\in \mathbb{Q}} W_n \textrm{ where for $w\in W_{n}$ we write ${\rm wt} w=n$}\label{uni}
	\end{align}
	such that
	\begin{align}
	&{\rm dim} \ W_{n}<\infty \textrm{ for $n \in \mathbb{Q} $ and} \nonumber \\
	&W_n=0 \textrm{ for $n$ sufficiently small}\label{doli}
	\end{align}
		together with a linear map
	\begin{align*}
	Y_W(\ . \ , z) : \ &V \longrightarrow ({\rm End} \ W)[[z,z^{-1}]]\\
	& v \longmapsto Y_W(v,z)
	\end{align*}
	such that $Y_W(v,z)$ is a field for each $v \in V$ and the following axioms hold:
	\begin{align*}
	Y_W(\vacuum, z)=Id_W
	\end{align*}
	\begin{align*}
	z^{-1}\dxyz Y_W(a,x)&Y_W(b,y)-z^{-1}\dmyxz Y_W(b,y)Y_W(a,x)=\\
	&y^{-1}\dxzy Y_W(Y(a,z)b,y),
	\end{align*}
		
	for all $a,b \in V$,
	\begin{align*}
	Y_W(\omega,z)=\dsum_{n\in \mathbb{Z}} \mathfrak{l}_W (n)z^{-n-2}
	\end{align*}
	
	satisfies 	
	\begin{align*}
	[\mathfrak{l}_W(m),\mathfrak{l}_W(n)]=(m-n)\mathfrak{l}_W(m+n)+\frac{1}{12}(m^3-m)\delta_{m+n,0}\rm{rank} V,
	\end{align*}
	
	\begin{align} \label{treli}
	{\mathfrak{l}_W}(0) w=({\rm wt } w) w \textrm{ \ \ \ for homogeneous $w\in W$}
	\end{align}
	and
	\begin{align*}
	\ddz Y_W(v,z)=Y_W(L(-1)v,z).
	\end{align*}
	
\end{definition}	
\begin{definition}
	A weak module $W$ for a vertex algebra $V$ satisfies all the axioms of a $V$-module except possibly (\ref{uni}), (\ref{doli}) or (\ref{treli}).
\end{definition}

\begin{definition}\cite{FLM}
Let $V$ be a vertex operator algebra and let $W$ be a $V$-module. The contragredient module $W^{'}:=\bigoplus_{n\in \mathbb Q}{\rm Hom}(W_n, \mathbb{C}) $ has the following $V$-module structure:	
\begin{align*}
\langle Y_{W^{'}}'(a,z)f,w \rangle=\langle f, Y_W(e^{xL(1)}(-x^{-2})^{L(0)}a, z^{-1})w \rangle,
\end{align*}
for $a\in V$, $f\in W'$, $m \in W$, where $\langle\ ,\ \rangle$ denotes the usual paring between a vector space and its dual vector space. 
\end{definition}
\begin{definition}
	A vertex algebra $V$ is called simple if it is irreducible as a $V$-module.
\end{definition}
\begin{definition}
	A vertex algebra $V$ is rational if it has finitely many non-isomorphic irreducible modules and every finitely generated module is a direct sum of irreducibles.
\end{definition}
 \subsubsection{The Vertex algebra $L(c,0)$.} For $c \in \mathbb{C}$ we first define the $Vir$-module $M_c$ as 

\begin{align*}
M_c:=  U(Vir^-) \tt _{U(Vir^\geq)} \mathbb{C}_{c} 	
\end{align*}
where
\begin{align*}
Vir^{\geq}:=\bigoplus_{n\geq-1}\mathbb{C}L_n \oplus \mathbb{C}C,
\end{align*}
\begin{align*}
Vir^{-}:=\bigoplus_{n\leq-2}\mathbb{C}L_n ,
\end{align*}
where $\mathbb{C}_c=\mathbb{C} \vac_0 $ is one dimenensional and its $Vir^{\geq}$-module structure is  given by $L_n \vac_0=0$ for $n\geq -1$ and $C\vac_0=c\vac_0$. Given the linear isomorphism $M_c\cong U(Vir^-)$ we use this identification to denote elements of $M_c$.

The vertex operator algebra structure in $M_c$ is given by:
\begin{enumerate}
\item[a)] Gradation
\begin{align*}
{\rm{wt}}(L_{-n_{1}}L_{-n_{2}}\cdots L_{-n_k}\vac_0)=  \dsum_{i=1}^{k} n_i, \ \ \ \textrm{  $n_1\geq n_2\geq\cdots \geq n_k\geq2,$} 
\end{align*}
\item[b)]
vacuum vector $ \vacuum=\vac_0,$ 
\\ \item[c)] translation operator $T:=L_{-1}$\\
\item[d)] conformal vector $\omega=L_{-2} \vac_0$, \\ 
\item [e)]vertex operators defined by
$Y(\vac_0, z):=Id$,

\begin{align}
Y(L_{-n_1 - 2} L_{-n_2 -2} \cdots L_{-n_k-2} \vac_0)=\frac{1}{n_1!\, n_2!\, \cdots \, n_k!}: \d^{n_1}l(z)\cdots\d^{n_k}l(z):
\end{align}
where $l(z)=\dsum_{n \in \mathbb{Z}} L_n z^{-n-2}$. 
\end{enumerate}
\begin{remark}\label{n+1} 
	Note that if we write $Y(\omega,z)=\dsum_{n \in \mathbb{Z}}\omega_n z^{-n-1} $, then $\omega_{n+1}=L_n$. 
\end{remark}

Using the description of singular vectors for all Verma modules given by Feigin and Fuchs \cite{FF} we have the following result which determines the vertex algebra structure of $L(c,0)$ for $c \in \mathbb{C}$:  
\begin{lemma}
	Let $c_{p,q}=1-\frac{6(p-q)^2}{pq}$ . Then,
	\begin{enumerate}
		\item[i)] $L(c, 0)=M_c$ if and only if $c\neq c_{p,q}$ for any  $p,q\geq 2$, $(p,q)=1$.
		\item[ii)] $L(c_{p,q}, 0)=M_{c_{p,q}}/\langle v_{p,q}\rangle$, where $v_{p,q}$ denotes the only (up to scalar) singular vector in $M_{c_{p,q}}$ and $\langle v_{p,q}\rangle$ denotes the submodule generated by this singular vector. 
	\end{enumerate}
\end{lemma}
Therefore, $L(c,0)$ has a vertex algebra structure given by that on $M_c$ when $c\neq c_{p,q}$ or by the quotient vertex algebra structure on $M_c/\langle v_{p,q}\rangle$ when $c=c_{p,q}$. 
\begin{remark}
It was conjectured in \cite{FZ} and proved in \cite{W} that the vertex Virasoro algebra $L(c,0)$ is rational if and only if $c=c_{p,q}$ for $p,q\geq 2$ and $(p,q)=1$. In particular, this implies that $L(25,0)$ is an irrational Virasoro vertex algebra. We will show in this work that, as in the case of $L(1,0)$, $L(25,0)$ admits a natural distinguished family of countably many irreducibles.	
\end{remark}
It was proved in \cite{FZ} that $M(c,h)$ is an $L(c,0)$-module with $Y_W=Y$.
 \subsubsection{The Zhu algebra and Intertwining operators} \label{222}
We need to recall some definitions and results from \cite{Z} and \cite{FZ} following the exposition from \cite{W}:

If $V$ is a vertex operator algebra, Zhu proved (with a slightly modified formula, as explained in Remark \ref{zhur}  ) that $V$ has an  algebra structure with the bilinear map $\ * \ $ determined by
\begin{align*}
a*b={\rm Res}_z \left( Y(a,z) \frac{(1-z)^{{\rm wt }a }}{z} b\right)
\end{align*} 
for homogeneous $a$ and for all $b$ $\in V$.
Also, recall that if $O(V)$ is defined as the linear span of elements of the form 
\begin{align*}
{\rm Res}_z\left(Y(a,z)\frac{(1-z)^{{\rm wt }a}}{z^2} b \right)
\end{align*}
for $a$, $b$ $\in V$ with $a$ homogeneous, then $O(V)\subset V$ is a two sided  ideal with the algebra strucute given by $*$.

\begin{definition}\cite{Z}
	The Zhu algebra of $V$ is defined as $A(V)=V/O(V)$.
\end{definition}
$A(V)$ is a unitary associative algebra for any vertex operator algebra $V$ and $[\omega]$ is in its center. (We will denote by $[a]$ the projection on $A(V)$ of an element $a \in V$). 
If $W$ is a $V$-module, we define the $A(V)$-module $A(W)$ with the analogous formulas:
First,  we set the following left and right action of $V$ on $W$ by
\begin{align*}
&a.w= {\rm Res}_z \left(Y(a,z)\frac{(1-z)^{{\rm wt }a}}{z} w \right)\\
&w.a={\rm Res}_z \left(Y(a,z)\frac{(1-z)^{{\rm wt }a-1}}{z} w \right),
\end{align*}
for any homogenous $a \in V$ and $w \in W.$
Next, we let $O(W)$ be the subspace of $W$ linearly spanned by all elements of the form
\begin{align*}
{\rm Res}_z \left(Y(a,z)\frac{(1-z)^{{\rm wt }a}}{z^2} w \right)
\end{align*}
for homogenous $a \in V$ and $w \in W$.
Finally, we define $A(W)=W/O(W)$. It was proved in \cite{FZ} that $A(W)$ is an  $A(V)$-bimodule.

\begin{lemma}\label{W0}\cite{Z1}
Let $V$ be a vertex operator algebra and let $W$ be a lowest weight $V$-module with lowest weight $h$ and with $L(0)$-weight decomposition $W=\bigoplus_{n\geq 0} W_{(n+h)}$. Define $W(n):=W_{(h+n)}$ and call $W(n)$ the degree $n$ subspace of $W$. 
Then, $W(0)$ is a natural $A(V)$-module with the action determined by the fact that $[a]$ acts on $W(0)$ by $(-1)^{{\rm deg}a}a_{{\rm wt}a-1}$ for any homogeneous $a \in V$.
\end{lemma}

\begin{lemma}\cite{Z} \label{muz}
	Let $V$ be a vertex operator algebra. Let $W_1$ be a submodule of the $V$-module $W$. Then, $A(W/W_1)\cong A(W)/[W_1]$, where $[W_1]$ denotes the image of $W_1$ under the projection $W\twoheadrightarrow A(W).$
\end{lemma}

For later use we recall the definition of intertwining operators:
\begin{definition}\cite{FHL}
	Let $W_1$, $W_2$ and $W_3$ be a triple of modules for a vertex operator algebra $V.$ An intertwining operator $\mathcal{Y}$ of type $\binom{W_3}{W_1 \ \ W_2}$ is a map
	\begin{align*}
	\mathcal{Y}:W_1 \tt W_2 \longrightarrow W_3\{z\}
	\end{align*}
that satisfies the following properties:
	
	1) Truncation: For any $w_i \in W_i,$ $i=1,2,$
	\begin{align*}
	(w_1)_n w_2=0,
	\end{align*}
	for $n>>0.$
	
	2) $L(-1)$-derivative: For any $w_1\in W_1$,
	\begin{align*}
	\mathcal{Y}(L(-1)w_1, z)=\ddz \mathcal{Y}(w_1,z),
	\end{align*}
	
	3) Jacobi identity: In Hom$(W_1\tt W_2, W_3)\{x,y,z\}$ we have
	\begin{align*}
	z^{-1}\dxyz Y_{W_3}(u,x)&\mathcal{Y}(w_1,y)-\dyxmz\mathcal{Y}(w_1,y)Y_{W_2}(u,x)=\\
	&y^{-1}\dxzy \mathcal{Y}(Y_{W_1}(u,z)w_1,y),
	\end{align*}
	for $u\in V$ and $w_1 \in W_1.$
\end{definition}
We denote the space of all the intertwining operators of the type $\binom{W_3}{W_1 \ \ W_2}$ by $I\binom{W_3}{W_1 \ \ W_2}.$

We recall the following important result due to \cite{FZ} and refined in \cite{L1}, \cite{L2}:

\begin{lemma}\label{Zhu}
	Let $V$ be a vertex operator algebra and let $W_1, W_2$ and $W_3$ be lowest weight $V$ modules. If $W_3$ is irreducible, then 
	\begin{align*}
	{\rm dim \ I} \binom{W_3}{W_1 \ \ W_2}\leq {\rm  dim \ Hom}_{A(V)}\left( A(W_1)\tt_{A(V)} W_2(0), W_3(0)\right). 
	\end{align*}		
\end{lemma}	
Finally, we mention the following lemma which will be useful in future computations. Its proof appears in \cite{HL} and \cite{L1}.
\begin{lemma}\label{HLL}
	Let $W_i$ for $i=1,2,3,$  be lowest weight modules for a vertex operator algebra $V$. Then 
\end{lemma}
\begin{align*}
{\rm dim} \ I\binom{W_3}{W_1, \ \ W_2}={\rm dim} \ I\binom{W_3}{W_2, \ \ W_1}={\rm dim} \ I\binom{W_2^{'}}{W_1, \ \ W_3^{'}}.
\end{align*}

\begin{definition} \cite{L2}\label{gv}
	Let $V$ be a vertex operator algebra, let $\hat{V}=V \tt \mathbb{C}[t, t^{-1}]$ and $d=L(-1)\tt 1+1\tt \frac{d}{dt}$. Then the lie algebra of $V$ is $g(V):=\hat{V}/d\hat{V}$ with bracket determined by  
	\begin{align*}
	[a(m), b(n)]=\dsum_{i=0}^\infty \binom{m}{i}(a_ib)(m+n-i),
	\end{align*}	
	where we denote $a(m)=a\tt t^m+dV.$
	If we set deg$a(m)={\rm {wt}} a-m-1$ then we have a triangular decomposition $g(V)=g(V)_-\oplus g(V)_0 \oplus g(V)_+.$
\end{definition}
If $U$ is a $g(V)_0$-module, then we define \cite{L2}
\begin{align*}
F(U)={\rm Ind}_{U(g(V)_+\oplus g(V)_0)}^{U(g)}U. 
\end{align*}
where $g(V)_+$ acts by $0$.
We define $\bar{F}(U)=F(U)/J(U)$ where $J(U)$ denotes the intersection of all the kernels of all $g(V)$-homomorphisms from $F(U)$ to weak modules.

 \begin{remark} \label{zhur}
 	In the original definition of the Zhu algebra \cite{Z} the product is given by 
 	\begin{align*}
 	a\bar{*}b={\rm Res}_z \left( Y(a,z) \frac{(1+z)^{{\rm wt }a }}{z} b\right)
 	\end{align*} 
 	for homogeneous $a$ and for all $b$ $\in V$
 	while $\bar{O}(V)$ is defined as the linear span of the elements of the form 
 	\begin{align*}
 	{\rm Res}_z\left(Y(a,z)\frac{(1+z)^{{\rm wt }a}}{z^2} b \right).
 	\end{align*}
 	In this case, an element $[\bar a]  \in \bar{A}(V):=V/\bar{O}(V)$ acts on the top level of a module $W$ as $a_{{\rm wt}a-1}$ for homogeneous $a\in V$ and everything in Section \ref{222} works in a parallel way. We work with the slightly modified formula presented in Section \ref{222} for convenience in the computations.
 \end{remark}
\section{The fusion rules for $L(25, 0)$}
Consider the vertex operator algebra $L(25,0)$ and let us denote by $\mathcal F_{25}$ its distinguished family of irreducible modules which are not Verma modules. Namely  $\mathcal F_{25}:=\{L(25, 1-\frac{n^2}{4})| \  n\geq 2\}.$
\begin{remark}
	Note that $L(25,1)$ and $L(25, 1-\frac{1}{4})$ are (irreducible) Verma modules and therefore, they do not belong in $\mathcal{F}_{25}$. 	
\end{remark}
Let $m, n, r\geq 2$. We want to describe the space of intertwining operators
\begin{align*}
I\binom{L(25, 1-r^2/4)}{ L(25, 1-m^2/4) \ \ L(25, 1-n^2/4) }.
\end{align*}

From the irreducibility of the non equivalent $L(25,0)$ modules $L(25,1-\frac{n^2}{4}), n \geq 2 $ it follows that

${\rm dim \ } I\binom{L\left(25, 1-\frac{k^2}{4}\right)}{ L(25, 0) \ \ L\left(25, 1-\frac{n^2}{4}\right) }=	 \begin{cases} 
1 & \textrm{\rm if } k=n  \\       
0 & \textrm{\rm otherwise}  
\end{cases}$

which together with Lemma \ref{HLL} implies that  

${\rm dim \ } I\binom{L\left(25, 1-\frac{k^2}{4}\right)}{ L\left(25, 1-\frac{n^2}{4}\right) \ \ L(25,0) }=	 \begin{cases} 
1 & \textrm{\rm if } k=n  \\       
0 & \textrm{\rm otherwise.}  
\end{cases}$

Therefore, we assume that $m,n\geq 3$ from now on.

We want to use Lemma \ref{Zhu} to bound the fusion rules. We first recall the $A(V)$ structure in the case of the Virasoro algebra of central charge 25:  From \cite{FZ} and \cite{W}, we know that
\begin{lemma}\label{w1}
	If $\omega\in L(25,0)$ is the conformal vector $\omega=L_{-2}\vac_{0}$, then there exists an isomorphism of associative algebras 
	\begin{align}\label{w1}
	A(L(25,0))&\cong \mathbb{C}[y] \\
	[\omega]^n&\mapsto y^n, \label{y2}
	\end{align}
	where $[\omega]=[(L_{-2}-L_{-1})\vac_{0}]$.		
\end{lemma}	

Moreover, as shown in \cite{FZ} we have the following description of the $\mathbb{C}[y]$-module structure for the module $A(W(25, h))$ for any $h\in \mathbb{C}$: 
\begin{lemma} \label{w2}
	As a $\mathbb{C}[y]$-module	
	\begin{align}
	A(M(25,h))\cong \mathbb{C}[x,y]
	\end{align}
	The isomorphism is given by
	\begin{align*}
	\mathbb{C}[x,y]&\cong A(M(25,h))\\
	x^ny^m&\mapsto [(L_{-2}-2L_{-1}+L_0)^n][\omega]^m[\vac_{h}] 
	\end{align*}

In this isomorphism the lowest weight vector is identified with $1\in \mathbb{C}[x,y]$ and the actions left and right actions of $C[y]$ are described by
\begin{align} \label{xy2}
y.p(x,y)=xp(x,y), \ \ \ \ \ p(x,y).y=yp(x,y)
\end{align}
for $p(x,y) \in \mathbb{C}[x,y].$

\end{lemma}

We are interested in Hom$\left(A(L(25,1-\frac{m^2}{4}))\tt_{\mathbb{C}[y]} L(25,1-\frac{n^2}{4})(0), L(1,1-\frac{r^2}{4})(0)\right)$ so we first analyze 
\begin{align*}
A\left(M\left(25,1-\frac{m^2}{4}\right)\right)\tt_{\mathbb{C}[y]} L\left(25,1-\frac{n^2}{4}\right)\big(0\big)
\end{align*}

It is easy to see that as a $\mathbb{C}[y]$ -module \begin{align}\label{1r2}
L\left(25,1-\frac{n^2}{4}\right)(0)\cong \mathbb{C}_{1-\frac{n^2}{4}}
\end{align}
where we identify the image of the lowest weight vector $\vac_{1-\frac{n^2}{4}}$ with $1$ and have that $y.1=\left(1-\frac{n^2}{4}\right)1$.

We recall the following important results:

\begin{lemma}\label{1}\cite{Z} In the quotient space $A(L(25,0))=L(25,0)/O(L(25,0))$,
	\begin{align}
	[(L(-1)-L(0))v]=0
	\end{align}
	for any $v\in L(25,0).$
\end{lemma}	
\
\begin{lemma}\cite{W} \label{wang}
	In $A(L(25,0))$ we have that for $n\geq 1, \ v\in L(25,0)$
	\begin{align}
	[L(-n)v]=[(((n-1)L(-2)-L(-1))+L(0))v]. \label{ind}
	\end{align}	
\end{lemma}

Using the method described by Milas in \cite{M} we get the following 
\begin{lemma}
	In $A(M(25,1-\frac{m^2}{4}))\tt_{\mathbb{C}[y]} L(25,1-\frac{n^2}{4})(0),$ 
	\begin{align}
	\left[L(-j_1)\cdots L(-j_k)v_{1-\frac{m^2}{4}}\right]=\prod_{r=1}^{k}\left(j_r\left(1-\frac{n^2}{4}\right)-y+\beta(r,k)\right)\left[v_{1-\frac{m^2}{4}}\right] \label{y12}\\
	=\prod_{r=1}^{k}\left(j_r\left(1-\frac{n^2}{4}\right)-x+\beta(r,k)\right)\left[v_{1-\frac{m^2}{4}}\right]\label{x2}
	\end{align}	
	where $\beta(r,k)=j_{r+1}+\cdots + j_k+ \left(1-\frac{m^2}{4}\right)$.
\end{lemma}	 
\begin{proof}
	By \ref{xy2} it is clear that the right hand side of (\ref{y12}) and (\ref{x2}) are equivalent so it is enough to prove that (\ref{y12}) holds.
	We prove it by induction on $k$:
	If $k=1$ then we have that 
	\begin{align} \label{k12}
	\left[L(-j)v_{1-\frac{m^2}{4}}\right]=\left[L(-j) v_{1-\frac{m^2}{4}} \tt 1\right]
	\end{align}
	Next, using Lemma \ref{wang} together with (\ref{y2}),  (\ref{k12}) and (\ref{1r2})  we have that 
	\begin{align}\label{k112}
	\left[L(-j)v_{1-\frac{m^2}{4}}\right]&=\left[v_{1-\frac{m^2}{4}} \tt \left[(j-1)y+L(0)\right]1 \right]\\
	&=\left[v_{1-\frac{m^2}{4}} \tt \left[j\left(1-\frac{n^2}{4}\right)-y+L(0)\right]1\right]\nonumber \\
	&=\left[\left[j\left(1-\frac{n^2}{4}\right)-y+L(0)\right] v_{1-\frac{m^2}{4}}\tt 1\right]\\
	&=\left(j\left(1-\frac{n^2}{4}\right)-y+1-\frac{m^2}{4}\right)\left[v_{1-\frac{m^2}{4}}\right]
	\end{align}
	For $k>1$, 
	\begin{align}
	&\left[L(-j_1)L(-j_2)   \cdots L(-j_k)v_{1-\frac{m^2}{4}}\right]=\nonumber \\   & \ \ \ \left(j_1\left(1-\frac{n^2}{4}\right)-x+1-\frac{m^2}{4}+j_2+j_3\cdots+j_k\right)\left[L(-j_2)\cdots L(-j_k)v_{1-\frac{m^2}{4}}\right]\label{tucu2}
	\end{align}
	by the same computation as in (\ref{k112}). Finally, using the induction hypotesis we obtain (\ref{y12}).
	
\end{proof}

Following the analysis \cite{M}, we note that the term in the right hand side of (\ref{y12}) coincides with $P(j_1, \cdots j_k)$ where 
\begin{align}
L(-j_1)\cdots L(-j_k)w_o=P(j_1, \cdots j_k)w_{j_1+\cdots j_k} \label{s22}
\end{align}
if we fix $\l=\frac{n^2}{4}-1$ and $\m=2-\frac{n^2}{4}-\frac{m^2}{4}-x$ for $\mathcal{D}_{\l, \m}$ the density module as in Definition \ref{F}. 

Next, we apply the theory of singular vectors from Section \ref{svff} to $M(25, 1-\frac{m^2}{4})$. 
From the work of Feigin and Fuchs \cite{FF} we know that  for each $m\geq 2$ there is a singular vector $v_{m-1}$ in $M(25, 1-\frac{m^2}{4})$ of weight $1-\frac{m^2}{4}+m-1$ such that if $\langle v_{m-1}\rangle$ denotes the submodule generated by $v_{m-1}$ then 
$M(25, 1-\frac{m^2}{4})/\langle v_{m-1}\rangle \cong L(25,1-\frac{m^2}{4} ) $. 
Note that in this case, $t=-1, \ \ p=m-1, \ \ q=1,\ \ \theta=1,$ in the notation of Section \ref{svff}
and (\ref{apq}) becomes 
\begin{align*}
&f_{m-1,1}^2\left(\frac{n^2}{4}-1,2-\frac{m^2+n^2}{4}-x, -1\right)=\\
&\dprod_{\substack{ -\frac{m-2}{2} \leq l\leq \frac{m-2}{2}  }}  \bigg(\left(2-\frac{m^2+n^2}{4}-x+A_{m-1,1}(l, 0)\right) \\ &\ \ \ \ \ \ \ \ \ \ \ \ \ \ \ \ \ \quad \times \left(2-\frac{m^2+n^2}{4}-x +A_{m-1,1}(-l, 0)\right)+4l^2-(ln)^2 \bigg).
\end{align*}
We have that 
\begin{align*}
A_{m-1, 1}(l,0)=\left(\frac{m}{2}-1+l \right)\left(\frac{m}{2}-l+1\right)=\left(\frac{m}{2}\right)^2-(1-l)^2
\end{align*}
while
\begin{align*}
A_{m-1, 1}(-l,0)=\left(\frac{m}{2}\right)^2-(1+l)^2,
\end{align*}
so that
\begin{align}
&f_{m-1,1}^2\left(\frac{n^2}{4}-1,2-\frac{m^2+n^2}{4}-x, -1\right)=\nonumber \\
&\prod_{\substack{ -\frac{m-2}{2} \leq l\leq \frac{m-2}{2}  }} \left(\left(2-\frac{n^2}{4}-x-(1-l)^2\right)\left(2-\frac{n^2}{4}-x -(1+l)^2
\right)+4l^2-(ln)^2 \right)= \nonumber \\
&\prod_{\substack{ -\frac{m-2}{2} \leq l\leq \frac{m-2}{2}  }} \left(\left(1-\frac{n^2}{4}-x-l^2+2l\right)\left(1-\frac{n^2}{4}-x  -l^2-2l
\right)+4l^2-(ln)^2 \right)= \nonumber \\
&\prod_{\substack{ -\frac{m-2}{2} \leq l\leq \frac{m-2}{2}  }}
\left(\left(1-\frac{n^2}{4}-x-l^2\right)^2-4l^2
+4l^2-(ln)^2 \right)=\nonumber \\
&\prod_{\substack{ -\frac{m-2}{2} \leq l\leq \frac{m-2}{2}  }}\left(1-\frac{n^2}{4}-x-l^2+ln\right)\left(1-\frac{n^2}{4}-x-l^2-ln\right)=\nonumber \\
&\prod_{\substack{ -\frac{m-2}{2} \leq l\leq \frac{m-2}{2}  }}\left(1-x-\left(\frac{n}{2}-l\right)^2\right)\left(1-x-\left(\frac{n}{2}+l\right)^2\right)=\nonumber \\
&\prod_{\substack{ -\frac{m-2}{2} \leq l\leq \frac{m-2}{2}  }}\left(1-x-\left(\frac{n}{2}-l\right)^2\right)^2=\nonumber \\
& \ \ \ \ \prod_{i\in \mathcal{L}_{m,n}}\left(x-\left (1-\frac{i^2}{4}\right)\right)^2,\label{opg}
\end{align}
where we denote
\begin{align}
\mathcal{L}_{m,n}:=\{m-n+2, m-n+4, \cdots, m+n-2\}.
\end{align}
\begin{remark}\label{pi} The central charge $c=25$ for the Virasoro vertex operator algebra is fundamental to obtaining the simplified formulas for the projection of singular vectors on density modules. This central charge, which corresponds to $t=-1$ in our computations, is what allows us to get formula (\ref{opg}). Note that the dual central charge, $c=1$, corresponds to $t=1$ and that the singular vector formulas become simplified only for both of these dual values of $t$.  
\end{remark}

We get the following result 
\begin{lemma} \label{lilita} Let $n,m \geq 3$. Then
	as an $A(L(25, 0))$-module 
	\begin{align*}
	A\left(L \left(25, 1-\frac{m^2}{4}\right)\right)\tt_{A(L(25, 0))}L\left(25, 1-\frac{n^2}{4}\right)(0)
	\end{align*} is isomorphic to 	
	\begin{align}
	{\mathbb{C}[x]}\big/ _{\big\langle \dprod _{i\in \mathcal{L}_{m, n}} (x-(1-\frac{i^2}{4}) \big\rangle}.\label{ngm}
	\end{align}
\end{lemma}
\begin{proof}
	Let us recall that from (\ref{w1}), (\ref{w2}) together with  Remark \ref{1r2} we have 
	\begin{align}\label{somr}
	A\left(M\left(25,1-\frac{m^2}{4}\right)\right)\tt_{A(L(25,0))} L\left(25,1-\frac{n^2}{4}\right)\big(0\big)\cong \mathbb{C}[x,y]\tt_{\mathbb{C}[y]} \mathbb{C}_{1-\frac{n^2}{4}}.
	\end{align}
	Now, using projection formula for the singular vector in $v_{1,m-1}$ on $\mathcal{D}_{\frac{n^2}{4}-1, 2-\frac{n^2}{4}-\frac{m^2}{4}-x} $ that we obtained in (\ref{opg}) we see that for $m\geq 3$
	\begin{align}
	v_{1,m-1}w_0=\prod_{i \in \mathcal{L}_{m, n} } \left(x-\left(1-\frac{i^2}{4}\right)\right)w_{m-1}.
	\end{align} 
	Next, using Lemma \ref{muz} together with (\ref{somr}) and (\ref{s22}) we obtain that 
		\begin{align*}
	A\left(L \left(25, 1-\frac{m^2}{4}\right)\right)\tt_{A(L(25, 0))}L\left(25, 1-\frac{n^2}{4}\right)(0)\cong {\mathbb{C}[x]}\big/ _{\left\langle \prod _{i\in \mathcal{L}_{m, n}} (x-(1-i^2/4) \right\rangle}
	\end{align*}
	
\end{proof}

In particular, it follows from Lemma \ref{lilita} that if $m\geq n\geq 3$ then 
\begin{align*}
A\left(L \left(25, 1-\frac{m^2}{4}\right)\right)\tt_{A(25, 0)}L\left(25, 1-\frac{n^2}{4}\right)(0)\cong \bigoplus_{i\in \mathcal{L}_{m,n}}\mathbb{C}v_i
\end{align*}
where $\mathbb{C}v_i$ is the irreducible $A(L(25, 0))$-module such that $y.v_i=(1-\frac{i^2}{4})v_i$.

We have obtained the following result:
\begin{proposition} \label{prop} 
	For $m,n\geq 2$	
	\begin{align*}
	{\rm dim }  \ I\binom{L(25, 1-r^2/4)}{ L(25, 1-m^2/4) \ \ L(25, 1-n^2/4) }\leq 1.
	\end{align*}
	if $r\in \mathcal{L}_{m,n}$
	and 
	\begin{align*}
	{\rm dim }  \ I\binom{L(25, 1-r^2/4)}{ L(25, 1-m^2/4) \ \ L(25, 1-n^2/4) }=0
	\end{align*}
	for $r\notin \mathcal{L}_{m,n}$.
\end{proposition}
Let $g(L(25,0))$ be the lie algebra in (\ref{gv}) and let $\omega=L(-2)\vac\in L(25,0)$.
From Definition \ref{gv} we know that 
\begin{align*}
[\omega(m+1), \omega(n+1)]=(m-n)\omega(n+m+1)+ \delta_{m+n,0}\frac{25}{12}(m^3-m).
\end{align*}
This implies that for any $g(L(25,0))_0$-module $U$, $U(Vir_-)\tt U\hookrightarrow U(g(L(25,0)_-))\tt U\cong F(U)$. It follows that for any $h\in \mathbb{C}$, $M(25, h)\hookrightarrow F(M(25,h)(0))$ and that $M(25, h)^{'}\hookrightarrow F((M(25,h)(0))^{*})$. 

Assume that $m\leq n$. By replacing $F(M(25,1-\frac{n^2}{4})(0))$ with $M(25,1-\frac{n^2}{4})$ and $F(M(25,1-\frac{r^2}{4})(0)^*)$ with $M(25,1-\frac{r^2}{4})^{'}$ following the argument in \cite{L2} we obtain an intertwining operator of type 
\begin{align}\label{liy}
\binom{M(25, 1-\frac{r^2}{4})^{'}}{L(25, 1-\frac{m^2}{4}) \ \ M(25, 1-\frac{n^2}{4})}
\end{align}
that we denote by $\mathcal{Y}_1$.

The contragradient module $M(25, 1-\frac{r^2}{4})^{'}$ is not irreducible so it is not of lowest weight type so we have that $L(25, 1-\frac{r^2}{4})\cong L(25, 1-\frac{r^2}{4})^{'}\subset M(25, 1-\frac{r^2}{4})^{'}$. Namely, if $f_{0}\in$ Hom$(M(25, 1-\frac{r^2}{4})_0, \mathbb{C})\subset M(25, 1-\frac{r^2}{4})^{'} $ is such that $f_0(\vac_{25, 1-\frac{r^2}{4}})=1$, then 
\begin{align*}
U(Vir)f_0\cong L\left(25, 1-\frac{r^2}{4}\right).
\end{align*}
We want to get an intertwining operator of type  $\binom{L(25, 1-\frac{r^2}{4})}{L(25, 1-\frac{m^2}{4}) \ \ L(25, 1-\frac{n^2}{4})}$ from the intertwining operator obtained of type $\binom{M(25, 1-\frac{r^2}{4})^{'}}{L(25, 1-\frac{m^2}{4}) \ \ M(25, 1-\frac{n^2}{4})}$.
We first prove the following 
\begin{lemma} \label{jal}
	Let $m, n, r\geq3$ with $m\leq n\leq r$, and let $\mathcal{Y}_1$ be the intertwining operator obtained in (\ref{liy}). Let $v_{1-\frac{n^2}{4}}$ be the singular vector in $M(25, 1-\frac{n^2}{4})$ gene-\\rating the maximal submodule $J(25, 1-\frac{n^2}{4})\cong M(25, 1-\frac{(n-2)^2}{4})$, let $\vac_{25,1-\frac{r^2}{4}} \in M(25,1-\frac{r^2}{4})^{'\,' }(0)\cong M(25,1-\frac{r^2}{4})(0)  $ and $\vac_{25, 1-\frac{m^2}{4}} \in L(25, 1-\frac{m^2}{4})$ be the respective lowest weight vectors . Then, 
	\begin{align}
	\langle\vac_{25, 1-\frac{r^2}{4}},  \mathcal{Y}_1(\vac_{25, 1-\frac{m^2}{4}}, x)v_{1-\frac{n^2}{4}}\rangle=0 \label{sing0}.
	\end{align}
	
\end{lemma}
\begin{proof} 
	From the Jacobi identity it follows that 
	\begin{align*} 
	[L(-j), \mathcal{Y}_1(\vac_{25, 1-\frac{m^2}{4}}, x)]=&\dsum_{i\geq 0}\binom{j+i-2}{i}x^{-j-i+1}(-1)^i\mathcal{Y}_1(L(i-1)\vac_{25,1-\frac{m^2}{4}} , x)\\
	=&x^{-j+1}\mathcal{Y}_1(L(-1)\vac_{25,1-\frac{m^2}{4}} , x)-(1-j)x^{-j }\mathcal{Y}_1(L(0)\vac_{25, 1-\frac{m^2}{4}} , x) \\
	=&x^{-j+1}\ddy\mathcal{Y}_1(\vac_{25, 1-\frac{m^2}{4}} , x)+(1-j)\left(1-\frac{m^2}{4}\right)x^{-j }\mathcal{Y}_1(\vac_{25,1-\frac{m^2}{4}} , x) 
	\end{align*}
	which implies that for $j \in \mathbb{N}$
	\begin{align*}
	&\langle\vac_{25, 1-\frac{r^2}{4}}, \mathcal{Y}(\vac_{25,1-\frac{m^2}{4}},x)L(-j)\vac_{25,1-\frac{n^2}{4}}\rangle=\\
	&\langle\vac_{25, 1-\frac{r^2}{4}}, L(-j)\left(\mathcal{Y}(\vac_{25, 1-\frac{m^2}{4}},x)\vac_{25,1-\frac{n^2}{4}}\right)\rangle- \\ & \ \ \ \ \ \ \ \ \ \ \ \ \left( x^{-j+1}\ddx+ (1-j)\left(1-\frac{m^2}{4}\right) x^{-j}\right)\langle\vac_{25, 1-\frac{r^2}{4}}, \mathcal{Y}_1(\vac_{25, 1-\frac{m^2}{4}},x)\vac_{25,1-\frac{n^2}{4}}\rangle=\\
	&- \left( x^{-j+1}\ddx+ (1-j)\left(1-\frac{m^2}{4}\right) x^{-j}\right)\langle\vac_{25, 1-\frac{r^2}{4}}, \mathcal{Y}_1(\vac_{25,1-\frac{m^2}{4}},x)\vac_{25,1-\frac{n^2}{4}}\rangle,
	\end{align*}
	because (cf. \cite{FHL})
	\begin{align*}
	\langle\vac_{25, 1-\frac{r^2}{4}}, L(-j)\left(\mathcal{Y}_1(\vac_{25,1-\frac{m^2}{4}},x)\vac_{25,1-\frac{n^2}{4}}\right)\rangle=\langle L(j)\vac_{25, 1-\frac{r^2}{4}},\mathcal{Y}_1(\vac_{25,1-\frac{m^2}{4}},x)\vac_{25,1-\frac{n^2}{4}} \rangle=0.
	\end{align*}  
	
	More generally, 
	\begin{align}
	\langle&\vac_{25, 1-\frac{r^2}{4}}, \mathcal{Y}(\vac_{25, 1-\frac{m^2}{4}},x)L(-j_1)\cdots L(-j_k)\vac_{25,1-\frac{n^2}{4}}\rangle \nonumber \\
	&=\prod_{i=1}^k  -\left( x^{-j_i+1} \ddx + (1-j_i)x^{-j_i}\left(1-\frac{m^2}{4}\right)\right)\langle\vac_{25, 1-\frac{r^2}{4}}, \mathcal{Y}_1(\vac_{25, 1-\frac{m^2}{4}},x)\vac_{25,1-\frac{n^2}{4}}\rangle \nonumber\\
	&=\prod_{i=1}^k  -\left( x^{-j_i+1} \ddx + (1-j_i)x^{-j_i}\left(1-\frac{m^2}{4}\right)\right)C x^{\frac{m^2}{4}+ \frac{n^2}{4}-1-\frac{r^2}{4}} \nonumber \\
	&=(-1)^{k} \prod_{i=1}^{k} \left(  \frac{m^2}{4} + \frac{n^2}{4}-1-\frac{r^2}{4}  - \sum_{s= i+1}^{k} j_s + (1-j_i) \left(1-\frac{m^2}{4}\right)\right) C x^{\frac{m^2}{4}+ \frac{n^2}{4}-1-\frac{r^2}{4}- \sum_{i=1}^k j_i}  \nonumber \\
	&=   \prod_{i=1}^{k} \left( -j_i \frac{m^2}{4} +\frac{r^2}{4} + \sum_{s= i+1}^{k} j_s - \frac{n^2}{4} \right)C x^{  \frac{m^2}{4}+ \frac{n^2}{4}-1 -\frac{r^2}{4}- \sum_{i=1}^k j_i},  \label{fito}
	\end{align}
	where $C$ is a constant that depends on $\mathcal{Y}_1$ which we may assume to be equal to 1. Note that the coefficients in (\ref{fito}) coincide with the coefficients in (\ref{x2}) if we replace $x$ with $1-\frac{r^2}{4}$ and exchange the roles of $1-\frac{n^2}{4}$ and $1-\frac{m^2}{4}.$ Therefore, we have that $\langle \vac_{25, 1-\frac{r^2}{4}},  \mathcal{Y}_1(\vac_{25, 1-\frac{m^2}{4}}, z)v_{1-\frac{n^2}{4}}\rangle=0$ if and only if the corresponding projection in $A\left(L \left(25, 1-\frac{n^2}{4}\right)\right)\tt_{A(25, 0)}L\left(25, 1-\frac{m^2}{4}\right)(0)$ vanishes. We know from (\ref{ngm}) (with the roles of $m$ and $n$ exchanged) that
	\begin{align*}
	A\left(L \left(25, 1-\frac{n^2}{4}\right)\right)\tt_{A(25, 0)}L\left(25, 1-\frac{m^2}{4}\right)(0)\cong{\mathbb{C}[x]}\big/ _{\big\langle \dprod _{i\in \mathcal{L}_{m, n}} (x-(1-\frac{i^2}{4}) \big\rangle}
	\end{align*}
	Since $m\leq n$ we have that $\mathcal{L}_{n, m}\subset \mathcal{L}_{m, n}$. By hypothesis, a priory $r\in \mathcal{L}_{m, n}$ but using that  $r\geq n$ we have that $r\in \mathcal{L}_{n, m}.$ Therefore, we have that the projection in 
	$A\left(L \left(25, 1-\frac{n^2}{4}\right)\right)\tt_{A(25, 0)}L\left(25, 1-\frac{m^2}{4}\right)(0)$ vanishes which implies that (\ref{sing0}) holds.
	
\end{proof}

\begin{proposition}
	Let $3\leq m\leq n\leq r$, let $r\in \mathcal{L}_{m,n}$ and let $\mathcal{Y}_1$ be the intertwining operator obtained in (\ref{liy}) . Then, 
	\begin{align*}
	\langle w_3', \mathcal{Y}_1(w_1, x)w_2\rangle=0
	\end{align*}
	for any $ w_1 \in L(25, 1-\frac{m^2}{4}), w_2 \in M(25, 1-\frac{(n-2)^2}{4})\hookrightarrow M(25, 1-\frac{n^2}{4})$ and	$w_3^{'}\in M(25, 1-\frac{r^2}{4})^{' \,'}\cong M(25, 1-\frac{r^2}{4}).$
\end{proposition}
\begin{proof}
	It follows from the Jacobi identity together Lemma \ref{jal}.
\end{proof}
Let $\bar{\mathcal{Y}}_1$ be the operator defined by
\begin{align}
\bar{\mathcal{Y}}_1(w_1, x)w_2=\mathcal{Y}_1(w_1, x) \pi_n^{-1}(w_2)
\end{align}
where $\pi_n:M(25, 1-\frac{n^2}{4}) \twoheadrightarrow L(25, 1-\frac{n^2}{4}) $ is the canonical projection.
Then, $\bar{\mathcal{Y}}_1$ is a non trivial  intertwining operator of type  $\binom{M(25, 1-\frac{r^2}{4})^{'}}{L(25, 1-\frac{m^2}{4}) \ \ L(25, 1-\frac{n^2}{4})}$.
Next, we use that 
\begin{align}
I\binom{M(25, 1-\frac{r^2}{4})^{'}}{L(25, 1-\frac{m^2}{4}) \ \ L(25, 1-\frac{n^2}{4})}\cong I\binom{L(25, 1-\frac{n^2}{4})}{L(25, 1-\frac{m^2}{4}) \ \ M(25, 1-\frac{r^2}{4})}.\label{iso}
\end{align}
By our assumption, $3\leq m\leq n \leq r$ and $r\in \mathcal{L}_{m,n}$ so that $m-n+2\leq r \leq m+n-2$. It follows that $r-m+2\leq n \leq r+m-2 $, namely, $n\in \mathcal{L}_{r,m}$.

Let $\mathcal{Y}_2\in \binom{L(25, 1-\frac{n^2}{4})}{L(25, 1-\frac{m^2}{4}) \ \ M(25, 1-\frac{r^2}{4})}$ be the image of $\bar{\mathcal{Y}_1}$ under the isomorphism (\ref{iso}). Let $u_3^{'}\in L(25, 1-\frac{n^2}{4})^{'}\cong L(25, 1-\frac{n^2}{4})$ and let $u_1 \in L(25, 1-\frac{m^2}{4})$ then, for the lowest weight vector $\vac_{1-\frac{r^2}{4}}\in M(25, 1-\frac{r^2}{4})$ we have that
\begin{align}
\langle u&_3', \mathcal{Y}_2(\vac_{25, 1-\frac{m^2}{4}},x)L(-j_1)\cdots L(-j_k)\vac_{1-\frac{r^2}{4}}\rangle \nonumber \\
&=\prod_{i=1}^k  -\left( x^{-j_i+1} \ddx + (1-j_i)x^{-j_i}\left(1-\frac{m^2}{4}\right)\right)\langle u_3', \mathcal{Y}_2(\vac_{25, 1-\frac{m^2}{4}},x)\vac_{1-\frac{r^2}{4}}\rangle \nonumber\\
&=\prod_{i=1}^k  -\left( x^{-j_i+1} \ddx + (1-j_i)x^{-j_i}\left(1-\frac{m^2}{4}\right)\right)D x^{\frac{m^2}{4}+ \frac{r^2}{4}-1-\frac{n^2}{4}} \nonumber \\
&=(-1)^{k} \prod_{i=1}^{k} \left(  \frac{m^2}{4} + \frac{r^2}{4}-1-\frac{n^2}{4}  - \sum_{s= i+1}^{k} j_s + (1-j_i) \left(1-\frac{m^2}{4}\right)\right) D x^{\frac{m^2}{4}+ \frac{r^2}{4}-1-\frac{n^2}{4}- \sum_{i=1}^k j_i}  \nonumber \\
&=   \prod_{i=1}^{k} \left( -j_i \frac{m^2}{4} +\frac{n^2}{4} + \sum_{s= i+1}^{k} j_s - \frac{r^2}{4} \right)D x^{  \frac{m^2}{4}+ \frac{r^2}{4}-1- \frac{n^2}{4}- \sum_{i=1}^k j_i},  \label{fito2}
\end{align}
where $D$ is a constant that depends on $\mathcal{Y}_2$ which we may assume to be equal to 1. Note that again the coefficients in (\ref{fito2}) coincide with the coefficients in (\ref{x2}) if we replace $x$ with $1-\frac{n^2}{4}$ and exchange the roles of $1-\frac{r^2}{4}$ and $1-\frac{m^2}{4}.$ Therefore, if we denote by $v_{1-\frac{r^2}{4}}$ the singular vector generating the maximal submodule in $M(25, 1-\frac{r^2}{4})$, we have that $\langle\vac_{25, 1-\frac{n^2}{4}},  \mathcal{Y}_2(\vac_{25, 1-\frac{m^2}{4}}, z)v_{1-\frac{r^2}{4}}\rangle=0$ if and only if the corresponding projection in $A\left(L \left(25, 1-\frac{r^2}{4}\right)\right)\tt_{A(25, 0)}L\left(25, 1-\frac{m^2}{4}\right)(0)$ vanishes. We know from (\ref{ngm}) that
\begin{align*}
A\left(L \left(25, 1-\frac{r^2}{4}\right)\right)\tt_{A(25, 0)}L\left(25, 1-\frac{m^2}{4}\right)(0)\cong{\mathbb{C}[x]}\big/ _{\langle \prod _{i\in \mathcal{L}_{r, m}} (x-(1-i^2/4) \rangle}.
\end{align*}
Since by hypothesis $n\in \mathcal{L}_{r,m},$ it is clear that 
\begin{align*}
\langle u_3^{'},  \mathcal{Y}_2(u_1, x)v_{1-\frac{r^2}{4}} \rangle=0
\end{align*}
for any $u_3^{'}\in L(25, 1-\frac{n^2}{4})$ and $u_1\in L(25, 1-\frac{m^2}{4}).$ Hence, if we denote the canonical projection by $\pi_r: M(25, 1-\frac{r^2}{4})\twoheadrightarrow L(25, 1-\frac{r^2}{4})$ and define for $u_1 \in L(25, 1-\frac{m^2}{4}), u_2\in L(25, 1-\frac{r^2}{4})$
\begin{align*}
\bar{\mathcal{Y}}_2(u_1, x)u_2:=\mathcal{Y}_2(u_1,x)(\pi_r^{-1}(u_2))
\end{align*}
we obtain a non-trivial intertwining operator $\bar{\mathcal{Y}}_2$ of type 
\begin{align*}
\binom{L(25, 1-\frac{n^2}{4})}{L(25, 1-\frac{m^2}{4}) \ \ L(25, 1-\frac{r^2}{4})}.
\end{align*}
Finally, if $\mathcal{Y}_3$ is the image of $\bar{\mathcal{Y}}_2$ under the isomorphism in (\ref{iso}), then 
\begin{align*}
\mathcal{Y}_3\in I\binom{L(25, 1-\frac{r^2}{4})}{L(25, 1-\frac{m^2}{4}) \ \ L(25, 1-\frac{n^2}{4})}
\end{align*}
is a non trivial intertwining operator.

If instead, $3 \leq m\leq r \langle n$ using an analogous argument and the fact that $r\in \mathcal{L}_{n,m}$ we end up with an intertwining operator $\bar{\mathcal{Y}}^2$ of type  $\binom{L(25, 1-\frac{r^2}{4})}{L(25, 1-\frac{m^2}{4}) \ \ L(25, 1-\frac{n^2}{4})}$.

Using Lemma \ref{HLL}, Proposition \ref{prop} together with the discussion above we obtain the following result:

\begin{theorem}
	Let $m,n\geq 0$. Then 
	\begin{align*}
	{\rm{dim} \ I} \binom{L(25, 1-\frac{(r+2)^2}{4})}{L(25, 1-\frac{(m+2)^2}{4})\ \ L(25,1-\frac{(n+2)^2}{4})}=
	\begin{cases} 
	1 & \textrm{\rm if } r \in \{|m-n|, |m-n|+2, \cdots, m+n\}  \\       
	0 & \textrm{\rm otherwise.}  
	\end{cases}
	\end{align*}	
	
\end{theorem}
\begin{remark}
As we mentioned in Remark \ref{pi}, the central charge $c=25$ was critical to obtaining the correspondence of fusion rules for the family $\mathcal{F}_{25}$ of $L(25,0)$-modules and for the irreducible finite dimensional $sl(2,\mathbb{C})$-modules. The singular vector formulas become simplified when $t=1$ and $t=-1$ in (\ref{ct}) which correspond to dual the central charges $c=1$, studied in \cite{M},  and $c=25$ respectively.
\end{remark}
\bibliographystyle{amsalpha}

\end{document}